\documentclass[12pt, reqno]{amsart}
\usepackage{preamble}

\title[]{Maximal symplectic torus actions}
\author[]{Rei Henigman}
\address{School of Mathematical Sciences, Tel Aviv University
Ramat Aviv, Tel Aviv 69978, Israel}
\email{rei.henigman@gmail.com}

\begin{document}

\begin{abstract}
    There are several different notions of maximal torus actions on smooth manifolds, in various contexts: symplectic, Riemannian, complex. In the symplectic context, for the so-called isotropy-maximal actions, as well as for the weaker notion of almost isotropy-maximal actions, we give classifications up to equivariant symplectomorphism. These classification results give symplectic analogues of recent classifications in the complex and Riemannian contexts. Moreover, we deduce that every almost isotropy-maximal symplectic torus action is equivariantly diffeomorphic to a product of a symplectic toric manifold and a torus, answering a question of Ishida. The classification theorems are consequences of Duistermaat and Pelayo's classification of symplectic torus actions with coisotropic orbits.
\end{abstract}

\maketitle

\section{Introduction}\label{introduction_section}
Let a torus $T$ act on a closed connected smooth manifold $M$. We assume that all actions in this paper are smooth and effective, and that all manifolds are closed and connected. When $M$ is a symplectic manifold, the action is \textbf{symplectic} if it preserves the symplectic form. A symplectic torus action is \textbf{Hamiltonian} if it is generated by a momentum map, that is, a $T$-invariant map $\mu:M \rightarrow \mathfrak{t}^*$ which satisfies
\begin{equation*}
    -d\langle \mu, \xi\rangle = \iota_{X_{\xi}} \omega
\end{equation*}
for every $\xi$ in the Lie algebra $\mathfrak{t}$ of the torus $T$, where $X_{\xi}$ is the corresponding vector field on $M$.

The celebrated convexity theorem of Atiyah \cite{atiyah_convexity} and Guillemin-Sternberg \cite{gs_convexity} states that the image of a momentum map is a convex polytope in $\mathfrak{t}^*$. Building on this, Delzant \cite{delzant} gave a full classification of symplectic $2n$-manifolds with Hamiltonian $T^n$ actions, which are known as \textbf{symplectic toric manifolds}, up to equivariant symplectomorphism. His theorem inspired many equivariant classification results in the field. Delzant's classification was generalized by Lerman and Tolman \cite{lerman_tolman} to orbifolds, and by Karshon and Lerman \cite{lerman_karshon} to non-compact manifolds. Furthermore, Karshon and Tolman (\cite{centered_hamiltonians, tall_uniqueness, tall_existence}) classified symplectic $2n$-manifolds with Hamiltonian $T^{n-1}$ actions, up to equivariant symplectomorphism, under the assumption that all of the reduced spaces are of dimension two (see also \cite{karshon_periodic, ahara_hattori, audin_paper, orbifold_ham_s1}). Additionally, several equivariant classification theorems were proved for symplectic non-Hamiltonian torus actions, see, e.g., \cite{duistermaat_pelayo, benoist, pelayo_two_tori, classification_non_ham_s1}.
\\

A main theme in these classification results is that the higher the dimension of the torus (the number of symmetries) with respect to the dimension of the manifold, the easier it becomes to characterize the space. In this paper, we are interested in \textit{maximal} symplectic torus actions. There are several existing notions of maximality, and we explore them in the context of symplectic geometry.
\\

Restricting to some class of torus actions on manifolds of a fixed dimension, 
one can refer to an action as \textbf{dimension-maximal} if the dimension of the torus is the largest out of all such actions. For example, one may restrict to actions that preserve some geometric structure (e.g., symplectic, Riemannian, complex), and one may put some restriction on the manifolds (e.g., simply-connected, positively-curved).

Grove and Searle \cite{grove_searle} classified dimension-maximal isometric torus actions on positively-curved Riemannian manifolds. Subsequently, Galaz-Garc\'ia and Searle \cite{galaz_searle} formalized the \textit{Maximal Symmetry Rank Conjecture}, which describes the expected classification of dimension-maximal isometric torus actions on simply-connected non-negatively curved Riemannian manifolds.

In symplectic geometry, dimension-maximal actions are relatively straightforward to understand. When considering symplectic torus actions, the only dimension-maximal action is the natural action of a torus on itself. If one restricts to Hamiltonian actions, then the highest dimension of a Hamiltonian $T$ action on a symplectic manifold $(M, \omega)$ is $\dim T = \dim M/2$. Therefore, dimension-maximal Hamiltonian torus actions are precisely symplectic toric manifolds, which were classified by Delzant \cite{delzant}. If we consider dimension-maximal symplectic torus actions on simply-connected symplectic manifolds, we again get that these are exactly symplectic toric manifolds.
\\

Recently, interesting relations were established between dimension-maximal actions and isotropy-maximal actions, introduced by Ishida \cite{ishida} (also called \textit{slice-maximal} in \cite{galaz_kerin_radeschi_wiemeler, galaz_kerin_radeschi}). A torus action is called \textbf{isotropy-maximal} if there exists a point $x \in M$ which satisfies
\begin{equation}\label{maximal_equation}
    \dim T + \dim T_x = \dim M,
\end{equation}
where $T_x$ is the stabilizer of $x$ in $T$. The name isotropy-maximal comes from the fact that the dimension of an isotropy group $T_x$ is always bounded from above by $\dim M - \dim T$ (see Remark \ref{isotropy_maximal_inequality_remark}). For more information, see \cite[Section 2]{ishida} and \cite[Section 2.2]{escher_searle}. More generally, following \cite{escher_searle}, an action is called \textbf{almost isotropy-maximal} if there exists a point $x \in M$ which satisfies
\begin{equation}\label{almost_maximal_equation}
    \dim T + \dim T_x \ge \dim M - 1,
\end{equation}
where $T_x$ is the stabilizer of $x$ in $T$. In particular, every isotropy-maximal action is also almost isotropy-maximal. An almost isotropy-maximal action that is not isotropy-maximal is called \textbf{strictly almost isotropy-maximal} (see \cite{dong_escher_searle}).

Grove and Searle's classification \cite{grove_searle}, mentioned above, implies that an isometric action on a positively-curved Riemannian manifold is isotropy-maximal if and only if it is dimension-maximal. Inspired by this, Escher and Searle \cite{escher_searle} classified isotropy-maximal isometric torus actions on simply-connected non-negatively curved Riemannian manifolds, and proved that the maximal symmetry rank conjecture is equivalent to the assertion that every dimension-maximal action is isotropy-maximal (in their context). Their classification generalizes earlier results, see \cite{wiemeler, galaz_kerin_radeschi_wiemeler}. Recently, Dong, Escher, and Searle \cite{dong_escher_searle} generalized this classification from isotropy-maximal actions to almost isotropy-maximal actions.

In the complex category, Ishida \cite{ishida} gave a full classification of isotropy-maximal complex torus actions on complex manifolds, up to equivariant biholomorphism (see also \cite{ishida_karshon}). These torus actions generalize LVM manifolds \cite{lvm_1, lvm_2}, or more generally LVMB manifolds \cite{lvmb}. Ishida also showed that if a Kähler manifold admits an isotropy-maximal complex torus action, then it is equivariantly diffeomorphic to a product of a projective nonsingular toric variety and a torus. This led him to ask whether every symplectic manifold, equipped with an isotropy-maximal symplectic torus action, is diffeomorphic to a product of a projective nonsingular toric variety and a torus. See Problem 12.4 in \cite{ishida}.
\\

The main result of this paper, is a full classification of almost isotropy-maximal symplectic torus actions, up to equivariant symplectomorphism (see Theorems \ref{almost_uniqueness_theorem} and \ref{almost_existence_theorem}). This includes the special case of isotropy-maximal actions, for which we provide a simplified version of the classification (see Theorems \ref{maximal_uniqueness_theorem} and \ref{maximal_existence_theorem}). These classifications give symplectic analogues of the classifications of isotropy-maximal and almost isotropy-maximal actions in \cite{ishida, escher_searle, dong_escher_searle}. In particular, we deduce that every almost isotropy-maximal symplectic torus action is equivariantly diffeomorphic to a product of a symplectic toric manifold and a torus. This gives a positive answer to Ishida's question, and to its generalization to almost isotropy-maximal actions.
\\

The proof of the classification theorem is based on the following key proposition which we prove in Section \ref{maximal_is_coisotropic_proof_section}: every almost isotropy-maximal symplectic torus action has a coisotropic orbit (see Proposition \ref{maximal_is_coisotropic_proposition}). Our classification theorem then follows from Duistermaat and Pelayo's \cite{duistermaat_pelayo} classification of symplectic torus actions with coisotropic orbits. In Lemma \ref{which_coisotropic_are_maximal_lemma}, we give simple conditions on which symplectic torus actions with coisotropic orbits are isotropy-maximal or almost isotropy-maximal.
\\

Lastly, we strengthen the relation between isotropy-maximal actions and the problem of extending a torus action to a higher dimensional torus action. While isotropy-maximal actions can never be extended to higher dimensional torus actions (this follows from Equation \eqref{inequality_isotropy}), the other direction is not true in general. A trivial example is the trivial action of a zero-dimensional torus on a surface $\Sigma_g$ of genus $g \ge 2$, which is not isotropy-maximal, but cannot be extended since $\Sigma_g$ does not admit any circle action. For non-trivial examples, see Section \ref{examples_section}. If one restricts to Hamiltonian actions, Karshon \cite{karshon_periodic} constructed Hamiltonian $S^1$ actions in dimension four which do not extend to Hamiltonian $T^2$ actions. These are also not isotropy-maximal. Similarly, see \cite{extending_tall_complexity_one, extending_short_complexity_one} for extensions of Hamiltonian $T^{n-1}$ actions in dimension $2n$, to Hamiltonian $T^n$ actions.

On the contrary, for almost isotropy-maximal actions, we show that the other direction is always true. More precisely, from our classification, we deduce that every almost isotropy-maximal symplectic torus action can be extended to an isotropy-maximal action (see Corollary \ref{almost_isotropy_maximal_lifting_corollary}). This is analogous to a recent result of Dong-Escher-Searle in the context of simply-connected non-negatively curved Riemannian manifolds, see \cite[Theorem A]{dong_escher_searle}.

\subsubsection*{Acknowledgements}
I thank my advisor Yael Karshon for suggesting me this problem, and for very useful discussions, as always. Furthermore, I thank my advisor Leonid Polterovich for crucial comments and suggestions. Finally, I thank Mark Berezovik, Joé Brendel, and the anonymous referee for useful feedback on this manuscript.

This research is partially supported by an NSF-BSF Grant 2021730, and by an ISF-NSFC Joint Research Program Grant/Award Number 3231/23.

\section{Main results}\label{results_section}
Before stating the main results, we recall some definitions.
By the Atiyah-Guillemin-Sternberg convexity theorem (\cite{atiyah_convexity}, \cite{gs_convexity}), the image of a momentum map is a convex polytope. A \textbf{symplectic toric manifold} is a $2n$-dimensional symplectic manifold, equipped with a Hamiltonian $T^n$ action. By \cite{delzant}, symplectic toric manifolds are classified up to equivariant symplectomorphism, by their momentum polytopes, which are called \textbf{Delzant polytopes}.
\\

The following are examples of isotropy-maximal and almost isotropy-maximal symplectic torus actions.
\begin{example}\label{maximal_example_delzant}
    Let $(M, \omega, \mu)$ be a symplectic toric manifold. The action must have a fixed point, and therefore by Equation \eqref{maximal_equation}, the action is isotropy-maximal. As Ishida notes in his paper, every isotropy-maximal symplectic torus action that have at least one fixed point is a symplectic toric manifold (this follows from \cite[Theorem 3.13]{giacobbe} or \cite[Remark 2 (4)]{ishida_karshon}).
\end{example}
\begin{remark}
    In contrast to the case of symplectic toric manifolds, where the torus is of the maximal dimension $\dim T = \dim M/2$, a Hamiltonian $T$ action on a $2n$-dimensional symplectic manifold with $\dim T < \dim M/2$ is not even almost isotropy-maximal. This follows since $\dim T + \dim T_x$ is always bounded from above by $\dim M - 2$, and so Equation \eqref{almost_maximal_equation} can never be satisfied.
\end{remark}
\begin{example}\label{maximal_example_torus}
    Let $T^{2n}$ be a torus, given with a translation-invariant symplectic form $\omega$. Then the natural action of $T^{2n}$ on itself preserves $\omega$. Because the action is free, by Equation \eqref{maximal_equation}, the action is isotropy-maximal. Furthermore, if $H \subset T^{2n}$ is a subtorus of dimension $2n-1$, then its action on $T^{2n}$ is an almost isotropy-maximal symplectic torus action.
\end{example}
\begin{example}\label{product_example}
    Let $(M, \omega)$ and $(M', \omega')$ be symplectic manifolds, equipped with isotropy-maximal symplectic torus actions of $T$ and $T'$. Then $(M \times M', \omega + \omega')$ is a symplectic manifold, equipped with an isotropy-maximal symplectic torus action of $T \times T'$ that acts by $(t, t') \cdot (x, y) = (t \cdot x, t' \cdot y)$.
\end{example}

The orbits of a smooth torus action on a connected manifold $M$, which have the smallest stabilizer, are called \textbf{principal orbits}, and their union form an open and dense subset of $M$. Since our action is effective, an orbit is principal if and only if it is free. See \cite[Appendix B, Section 3.5]{ggk_book} for more information.

An orbit $\mathcal{O}$ of a symplectic torus action on a symplectic manifold $(M, \omega)$ is called \textbf{coisotropic} if for every point $x \in \mathcal{O}$, the tangent space $T_x\mathcal{O}$ of the orbit $\mathcal{O}$ is a coisotropic subspace of $T_x M$. We make the observation that all of the above examples of almost isotropy-maximal actions have coisotropic principal orbits. We prove that this is in fact always the case for almost isotropy-maximal symplectic torus actions.
\begin{proposition}\label{maximal_is_coisotropic_proposition}
    Let $(M, \omega)$ be a symplectic manifold, equipped with an almost isotropy-maximal symplectic torus action. Then the principal orbits of the torus action are coisotropic.
\end{proposition}
\begin{remark}
    Proposition \ref{maximal_is_coisotropic_proposition} is tight in the following sense. There exist symplectic torus actions with no coisotropic orbits, satisfying the equality
    \begin{equation*}
        \dim T + \dim T_x = \dim M - 2
    \end{equation*}
    at some point $x \in M$. For example, every Hamiltonian $T^{n-1}$ action on a $2n$-dimensional symplectic manifold satisfies this equality at every fixed point, and has no coisotropic orbits. Another example is the free action of $T^2$ on $(T^4, dp_1\wedge dq_1 + dp_2\wedge dq_2)$, which rotates the $p_1$ and $q_1$ coordinates. See also \cite{pelayo_two_tori} for the list of all symplectic $T^2$ actions on $4$-manifolds, all of whose orbits are symplectic (and therefore have no coisotropic orbits).
\end{remark}
\begin{remark}
    By \cite[Proposition 5.1]{benoist}, a symplectic torus action has at least one coisotropic principal orbit if and only if every principal orbit is coisotropic.
\end{remark}
We prove Proposition \ref{maximal_is_coisotropic_proposition} in Section \ref{maximal_is_coisotropic_proof_section}. The main ingredient in the proof is the unique antisymmetric bilinear form $\sigma$ defined on the Lie algebra $\mathfrak{t}$ of the acting torus $T$, introduced in \cite{benoist, duistermaat_pelayo}, see the definition below (Equation \eqref{bilinear_form_equation}).
\\

In light of Proposition \ref{maximal_is_coisotropic_proposition}, we can apply tools from Duistermaat and Pelayo's work \cite{duistermaat_pelayo} on the classification of symplectic torus actions with coisotropic principal orbits (see also \cite{benoist}). In their work, they define six invariants for symplectic torus actions with coisotropic principal orbits, and show that two spaces are equivariantly symplectomorphic if and only if their six invariants agree. Moreover, they construct a space for every choice of values for the six invariants. In the following paragraphs, we describe the first three invariants of their work, which are relevant for the statements of Lemma \ref{which_coisotropic_are_maximal_lemma}, Theorem \ref{maximal_uniqueness_theorem}, and Theorem \ref{maximal_existence_theorem}. Later in this section, we give a brief description of the other three invariants. For more details see \cite[Section 9]{duistermaat_pelayo}.
\\

Let $T$ act symplectically on a symplectic manifold $(M, \omega)$, with coisotropic principal orbits. The \textbf{first invariant} in the Duistermaat-Pelayo classification is the antisymmetric bilinear form $\sigma$ on the Lie algebra $\mathfrak{t}$ of the torus $T$, given by
\begin{equation}\label{bilinear_form_equation}
    \sigma(X, Y) = \omega_p(X_M, Y_M),
\end{equation}
for any arbitrary point $p \in M$, and every $X, Y \in \mathfrak{t}$, where $X_M, Y_M \in \mathfrak{X}(M)$ are the corresponding vector fields on $M$. The bilinear form $\sigma$ is well defined by \cite[Lemma 2.1]{duistermaat_pelayo} (see also \cite[Lemme 2.1]{benoist} and \cite[Section 3]{ginzburg_symplectic_actions}).

The \textbf{second invariant} in the Duistermaat-Pelayo classification is the (unique) maximal subtorus $T_h$ of $T$ which acts in a Hamiltonian fashion (see \cite[Lemma 3.6]{duistermaat_pelayo}). Its Lie subalgebra $\mathfrak{t}_h$ is contained in the kernel of the form $\sigma$, defined by
\begin{equation*}
    \ker(\sigma) := \{X \in \mathfrak{t}\ \mid\ \forall\ Y \in \mathfrak{t}\colon\  \sigma(X, Y) = 0\}.
\end{equation*}

The \textbf{third invariant} in the Duistermaat-Pelayo classification is the image $\Delta \subset \mathfrak{t}_h^*$ of a momentum map $\mu:M \rightarrow \mathfrak{t}_h^*$ that generates the Hamiltonian action of $T_h$ on $M$, where we choose $\mu$ such that the center of mass of its image $\Delta$ is the origin. By \cite[Corollary 3.11]{duistermaat_pelayo}, the image $\Delta$ is a Delzant polytope.
\\

By Proposition \ref{maximal_is_coisotropic_proposition}, every almost isotropy-maximal symplectic $T$ action has coisotropic principal orbits. Hence, every almost isotropy-maximal symplectic $T$ action has well defined values for the six invariants, and in particular, well defined bilinear form $\sigma$, subtorus $T_h \subset T$, and Delzant polytope $\Delta \subset \mathfrak{t}_h^*$.

To deduce a classification theorem for almost isotropy-maximal symplectic torus actions from Duistermaat and Pelayo's classification, we need to determine which symplectic torus actions with coisotropic principal orbits are almost isotropy-maximal. This is the content of the following lemma, which we prove in Section \ref{lemma_proof_section}.
\begin{lemma}\label{which_coisotropic_are_maximal_lemma}
    Let $(M, \omega)$ be a symplectic manifold, equipped with a symplectic torus action with coisotropic principal orbits. Then the action is almost isotropy-maximal if and only if $\dim\mathfrak{t}_h \ge \dim\ker(\sigma) - 1$. Moreover, it is isotropy-maximal if and only if $\mathfrak{t}_h = \ker(\sigma)$.
\end{lemma}

In the special case of isotropy-maximal actions, the classification takes a particularly simple form. To be precise, in this case we prove that the values of the six Duistermaat-Pelayo invariants are determined only by the first invariant $\sigma$ and the third invariant $\Delta$. Therefore, $\sigma$ and $\Delta$ alone determine whether two isotropy-maximal symplectic torus actions are equivariantly symplectomorphic.
\begin{theorem}\label{maximal_uniqueness_theorem}
    Let $(M, \omega)$ and $(M', \omega')$ be symplectic manifolds, equipped with isotropy-maximal symplectic torus actions. Then they are equivariantly symplectomorphic if and only if their bilinear forms $\sigma$ and $\sigma'$ and their Delzant polytopes $\Delta$ and $\Delta'$ are the same.
\end{theorem}
Furthermore, we give a simple construction for isotropy-maximal symplectic torus actions as products of symplectic toric manifolds and tori.

\begin{theorem}\label{maximal_existence_theorem}
    Let $T$ be a torus with Lie algebra $\mathfrak{t}$, and let $T_h$ be a subtorus with Lie subalgebra $\mathfrak{t}_h$. Let $\sigma$ be an antisymmetric bilinear form on $\mathfrak{t}$ with $\ker(\sigma) = \mathfrak{t}_h$. Moreover, let $\Delta \subset \mathfrak{t}_h^*$ be a Delzant polytope with center of mass at the origin, and let $(M_h, \omega_h, T_h)$ be a symplectic toric manifold with momentum image $\Delta$. Choose a complementary subtorus $T_f$ to $T_h$ in $T$, and let $\omega_f$ be the symplectic form on $T_f$ that is induced from $\sigma$. 
    
    Then the action of $T \cong T_h \times T_f$ on the product $(M_h \times T_f, \omega_h + \omega_f)$ is an isotropy-maximal symplectic action with bilinear form $\sigma$ and Delzant polytope $\Delta$.
\end{theorem}
Together, Theorems \ref{maximal_uniqueness_theorem} and \ref{maximal_existence_theorem} give a full classification of isotropy-maximal symplectic torus actions up to equivariant symplectomorphism. See Section \ref{proofs_of_class_section} for the proofs.

From Theorems \ref{maximal_uniqueness_theorem} and \ref{maximal_existence_theorem}, we immediately deduce that every isotropy-maximal symplectic torus action is equivariantly symplectomorphic to a product of a symplectic toric manifold and a torus, giving a positive answer to Ishida's question (Problem 12.3 in \cite{ishida}).
\\

Before stating the classification of almost isotropy-maximal symplectic torus actions, we first discuss the last three invariants in the classification of actions with coisotropic orbits in \cite{duistermaat_pelayo}. A reader might choose to skip the following paragraphs and jump straight to the statements of Theorems \ref{almost_uniqueness_theorem} and \ref{almost_existence_theorem} in their first reading.

Let $\evmap:H_1(M, \mathbb{Z}) \rightarrow \ker(\sigma)^*$ be the map given by $\evmap(A)(\xi) := \int_A \iota_{X_{\xi}} \omega$, where $X_{\xi} \in \mathfrak{X}(M)$ is the vector field on $M$ corresponding to $\xi \in \ker(\sigma)$. Then the \textbf{fourth invariant} in the Duistermaat-Pelayo classification is the subgroup of periods $P := \image(\evmap)$ of $\ker(\sigma)^*$. Because $\iota_{X_{\xi}} \omega$ is exact whenever $\xi$ is in $\mathfrak{t}_h$, we can view $P$ as a subgroup of the vector space $N :=(\ker(\sigma) /\mathfrak{t}_h)^*$. By \cite[Proposition 3.8]{duistermaat_pelayo}, $N/P$ is compact, and there exists a $T$-invariant map $\mu:M \rightarrow \Delta \times N/P$ which induces a homeomorphism $\bar \mu:M/T \rightarrow \Delta \times N/P$. If $\exp(\ker(\sigma))$ happens to be a subtorus $H$ of $T$, the map $\mu$ is just a cylinder-valued momentum map for the $H$-action on $M$ (see for example \cite{cylinder_momentum_map} or \cite[Section 5.2]{ratiu_ortega_book}).

Let $M_{\free} \subset M$ be the union of free orbits in $M$, and let $\mu_{\free}:M_{\free} \rightarrow \Delta \times N/P$ be the restriction of $\mu$ to $M_{\free}$. Following \cite[Remark 5.7]{duistermaat_pelayo}, the image of the map $\mu_{\free}$ is $\interior(\Delta) \times N/P$, and the map $\mu_{\free}:M_{\free} \rightarrow \interior(\Delta) \times N/P$ defines a principal $T$-bundle. The \textbf{fifth invariant} in the Duistermaat-Pelayo classification is the Chern class of this bundle. This Chern class is an element of $H^2(\interior(\Delta) \times N/P, \mathfrak{t})$, and by \cite[Proposition 5.5 and Remarks 5.7--5.8]{duistermaat_pelayo} can be identified with an antisymmetric bilinear map $c:N \times N \rightarrow \ker(\sigma) \cap T_{\mathbb{Z}}$ satisfying
\begin{equation}\label{equation_fifth_inv}
    \xi(c(\xi', \xi'')) + \xi'(c(\xi'', \xi)) + \xi''(c(\xi, \xi')) = 0,
\end{equation}
for every $\xi, \xi', \xi'' \in P$. Here, the integral lattice $T_{\mathbb{Z}} \subset \mathfrak{t}$ is the kernel of the exponential map.

From the holonomy of a particular kind of a connection (see \cite[Definition 5.3]{duistermaat_pelayo}) on the principal bundle $M_{\free} \rightarrow \interior(\Delta) \times N/P$, Duistermaat and Pelayo obtain an element $\tau$ of the space of mappings
\begin{equation*}
    \Hom_c(P, T) := \left\{\tau:P \rightarrow T|\ \tau(\xi)\tau(\xi') = \tau(\xi + \xi')e^{c(\xi, \xi')/2} \text{ for all }\xi,\xi' \in P\right\},
\end{equation*}
which is unique up to an action of a certain group $\exp(\mathcal{A})$ on $\Hom_c(P, T)$ (see \cite[Definition 7.10]{duistermaat_pelayo}). The \textbf{sixth invariant} in the Duistermaat-Pelayo classification is the class $\bar \tau$ of $\tau$ in $\Hom_c(P, T) / \exp(\mathcal{A})$. For the complete details on the connections, the holonomy invariant, and the definition of $\exp(\mathcal{A})$, see Sections 5 and 7 in \cite{duistermaat_pelayo}.

Let $T$ be a torus. Following \cite[Definition 9.1]{duistermaat_pelayo}, a \textbf{list of ingredients} $(\sigma, T_h, \Delta, P, c, \bar \tau)$ for the torus $T$ consists of a bilinear form $\sigma$ on $\mathfrak{t}$, a subtorus $T_h$ with $\mathfrak{t}_h \subset \ker(\sigma)$, a Delzant polytope $\Delta \subset \mathfrak{t}_h^*$ centered at the origin, a cocompact subgroup $P$ of $N:=(\ker(\sigma)/\mathfrak{t}_h)^*$, an antisymmetric bilinear map $c:N \times N \rightarrow \ker(\sigma) \cap T_{\mathbb{Z}}$ satisfying \eqref{equation_fifth_inv}, and an equivalence class $\bar \tau$ in $\Hom_c(P, T) / \exp(\mathcal{A})$. This definition will be useful for statement of Theorem \ref{almost_existence_theorem} (c.f. \cite[Theorem 9.6]{duistermaat_pelayo}).
\\

Now, we are ready to state the classification of almost isotropy-maximal symplectic torus actions. By Proposition \ref{maximal_is_coisotropic_proposition} and \cite[Theorems 9.4]{duistermaat_pelayo}, we immediately deduce that two almost isotropy-maximal symplectic torus actions are equivariantly symplectomorphic if and only if their six invariants agree.
\begin{theorem}\label{almost_uniqueness_theorem}
    Let $(M, \omega)$ and $(M', \omega')$ be symplectic manifolds, equipped with almost isotropy-maximal symplectic $T$ actions. Then they are equivariantly symplectomorphic if and only if their six invariants in the Duistermaat-Pelayo classification \cite[Definition 9.3]{duistermaat_pelayo} agree.
\end{theorem}
\begin{remark}\label{fifth_invariant_remark}
    The fifth invariant $c:N \times N \rightarrow \ker(\sigma)$ is always trivial for almost isotropy-maximal actions, as can be seen by the proofs of Theorem \ref{maximal_uniqueness_theorem} and Theorem \ref{almost_existence_theorem}. Thus, only the other five invariants are needed to determine whether two actions are isomorphic. Note that the other five invariants are not trivial for almost isotropy-maximal actions.
\end{remark}
To complete the classification, we describe a simple construction for strictly almost isotropy-maximal symplectic torus actions as products of symplectic toric manifolds and tori. Together with Theorem \ref{maximal_existence_theorem}, this gives a simple construction for all almost isotropy-maximal symplectic torus actions.
\begin{theorem}\label{almost_existence_theorem}
    Let $(\sigma, T_h, \Delta, P, c, \bar\tau)$ be a list of ingredients as in Duistermaat and Pelayo’s classification \cite[Definition 9.1]{duistermaat_pelayo} (reviewed above), which satisfies $\dim \mathfrak{t}_h = \dim\ker(\sigma) - 1$. Moreover, let $(M_h, \omega_h, T_h)$ be a symplectic toric manifold with momentum image $\Delta$.
    
    Then there exists a torus $\mathbb{T}$ with a $\mathbb{T}$-invariant symplectic form $\omega_{\mathbb{T}}$, a codimension one subtorus $K \subset \mathbb{T}$ acting on $\mathbb{T}$ by the natural inclusion, and an isomorphism $T \cong T_h \times K$, such that the action of $T \cong T_h \times K$ on the product $(M_h \times \mathbb{T}, \omega_h + \omega_{\mathbb{T}})$ is a strictly almost isotropy-maximal symplectic action, and its values of the Duistermaat-Pelayo invariants are $\sigma, T_h, \Delta, P, c,$ and $\bar\tau$.
\end{theorem}
See Section \ref{proof_of_almost_existence_theorem_section} for the proof of Theorem \ref{almost_existence_theorem}. Together, Theorems \ref{almost_uniqueness_theorem} and \ref{almost_existence_theorem} give a full classification of almost isotropy-maximal symplectic torus actions up to equivariant symplectomorphism. It immediately follows from Theorems \ref{almost_uniqueness_theorem} and \ref{almost_existence_theorem} that every strictly almost isotropy-maximal action is equivariantly symplectomorphic to a product of a symplectic toric manifold and a torus. This gives a positive answer to Ishida's question (Problem 12.3 in \cite{ishida}) in the more general setting of almost isotropy-maximal actions.
\begin{remark}
    The result that every almost isotropy-maximal action is equivariantly symplectomorphic to a product of a symplectic toric manifold and a torus is tight in the following sense. There exist non-trivial symplectic torus actions satisfying the equality
    \begin{equation*}
        \dim T + \dim T_x = \dim M - 2
    \end{equation*}
    at some point $x \in M$, that are not diffeomorphic to a product of a symplectic toric manifold and a torus. In particular, these actions cannot be extended to higher dimensional symplectic torus actions (because then they would be almost isotropy-maximal, and thus by Theorems \ref{almost_uniqueness_theorem}, \ref{almost_existence_theorem}, or by Theorems \ref{maximal_uniqueness_theorem}, \ref{maximal_existence_theorem}, they would be diffeomorphic to such a product). See Section \ref{examples_section} for several such examples.
\end{remark}
\begin{remark}
    The following is a short discussion of how the Duistermaat-Pelayo invariants can be expressed in terms of the decomposition $(M_h \times \mathbb{T}, \omega_h + \omega_{\mathbb{T}})$ given in Theorem \ref{almost_existence_theorem}.
    First, $T_h$ is explicitly given as the torus acting on $M_h$, and $\Delta$ is the momentum image of this action, centered at the origin. Next, by Remark \ref{fifth_invariant_remark}, the invariant $c$ is trivial. The other three invariants $\sigma, P,$ and $\bar \tau$ are determined by the two-form $\omega_{\mathbb{T}}$, by the choice of isomorphism $T \cong T_h \times K$, and by the embedding of $K$ into $\mathbb{T}$. The inclusion maps of $K$ into $T$ and $\mathbb{T}$ fit in the following commuting diagram.
    \[
    \begin{tikzcd}[column sep=large, row sep=large]
    \liealg(\mathbb{T}) \arrow[d, "\exp"'] & \liealg(K) \arrow[l, hook', "j"'] \arrow[r, hook, "i"] \arrow[d, "\exp"] & \mathfrak{t} \arrow[d, "\exp"] \\
    \mathbb{T} & K \arrow[l, hook', "\tilde{j}"'] \arrow[r, hook, "\tilde{i}"] & T
    \end{tikzcd}
    \]

    $\sigma$ is the unique two-form on $\mathfrak{t}$ with $\mathfrak{t}_h \subset \ker(\sigma)$, such that its pullback $i^* \sigma$ to $\liealg(K)$ agrees with the pullback $(\exp \circ j)^* \omega_{\mathbb{T}}$ of $\omega_{\mathbb{T}}$ to $\liealg(K)$. Denote this pullback by $\omega_{\mathbb{T}}|_{\liealg(K)}:= (\exp \circ j)^* \omega_{\mathbb{T}} \in \Omega^2(\liealg(K), \mathbb{R})$.
    
    The invariant $P$ is by definition the image of the map $\evmap:H_1(M, \mathbb{Z}) \rightarrow (\ker(\sigma)/\mathfrak{t}_h)^*$ defined by $\evmap(A)(\xi) = \int_A\iota_{X_{\xi}} \omega$ where $X_{\xi} \in \mathfrak{X}(M)$ is the vector field corresponding to $\xi \in \ker(\sigma)$. We can use the decomposition $(M, \omega) \cong(M_h \times \mathbb{T}, \omega_h + \omega_{\mathbb{T}})$ to identify the vector space $N : = (\ker(\sigma)/\mathfrak{t}_h)^*$ with $\ker(\omega_{\mathbb{T}}|_{\liealg(K)})$ in $\liealg(K)$. It follows that $P := \image(\evmap)$ is equal to the image of the map $\evmap':H_1(\mathbb{T}, \mathbb{Z}) \rightarrow \ker(\omega_{\mathbb{T}}|_{\liealg(K)})^*$ defined by $\evmap'(A)(v) = \int_A\iota_{X_v} \omega_{\mathbb{T}}$ where $X_v \in \mathfrak{X}(\mathbb{T})$ is the vector field corresponding to $v \in \ker(\omega_{\mathbb{T}}|_{\liealg(K)})$. By choosing a basis for $\mathbb{T}$, and a corresponding basis for $H_1(\mathbb{T}, \mathbb{Z})$, we can write $A, \omega_{\mathbb{T}}$, and $v$ in coordinates, and then we have $\evmap'(A)(v) = A^T \omega_{\mathbb{T}} v$, which gives an explicit description for $P$ in coordinates.
    
    In a similar fashion, one can describe the last invariant $\bar \tau$ explicitly, but it is rather technical.
    
\end{remark}

Lastly, from Theorems \ref{almost_uniqueness_theorem} and \ref{almost_existence_theorem}, we immediately deduce that a strictly almost isotropy-maximal symplectic torus action always extends to an higher dimensional symplectic torus action. Note that an extension of a strictly almost isotropy-maximal action must be isotropy-maximal.
\begin{corollary}\label{almost_isotropy_maximal_lifting_corollary}
    Let $(M, \omega)$ be a symplectic manifold, equipped with a strictly almost isotropy-maximal symplectic torus action. Then the action extends to an isotropy-maximal symplectic torus action.
\end{corollary}
This gives an analogue to a similar result in the context of simply-connected non-negatively curved Riemannian manifolds, see \cite[Theorem A]{dong_escher_searle}.

\section{Actions that cannot be extended}\label{examples_section}
In this section, we present examples of non-trivial symplectic torus actions, that cannot be extended to higher dimensional symplectic torus actions, but are not isotropy-maximal.  These are based on results from \cite{classification_non_ham_s1, tolman_non_kahler, pelayo_two_tori, karshon_periodic}. Moreover, we explain why isotropy-maximal torus actions cannot be extended to higher dimensional torus actions.
\begin{example}\label{non_ham_example}
    In the end of Section 2 of \cite{classification_non_ham_s1}, we gave an example of a symplectic non-Hamiltonian $S^1$ action on a four-dimensional manifold $(M, \omega)$, which does not extend to a symplectic $T^2$ action. Since $\dim M = 4$ and $\dim T = 1$, Equation \eqref{almost_maximal_equation} cannot be satisfied, so the action is not almost isotropy-maximal.
\end{example}
\begin{example}\label{example_kahler}
    Tolman \cite{tolman_non_kahler} constructed a Hamiltonian $T^2$ action on a six-dimensional symplectic manifold which does not admit an equivariant Kähler structure. Since her manifold is simply-connected, every symplectic torus action on it is Hamiltonian. However, symplectic toric manifolds always admit equivariant Kähler structures, and therefore her action cannot be extended to a Hamiltonian $T^3$ action, or more generally to a symplectic $T^3$ action. On the other hand, since $\dim T^2 < \dim M/2$, Equation \eqref{almost_maximal_equation} cannot be satisfied, so the action is not almost isotropy-maximal.    
\end{example}
\begin{example}\label{example_s2_t2}
    Let $S^1$ act on $(S^2 \times \Sigma_g, dh \wedge d\theta + \omega_{\Sigma})$ by rotating the sphere at speed $1$, where $\Sigma_g$ is some surface of genus $g > 1$, and $\omega_{\Sigma}$ is any symplectic form on $\Sigma_g$. This is a Hamiltonian action, generated by the height function $H = h$ on the sphere. As before, the action is not almost isotropy-maximal. By \cite[Theorem 1.1 or Theorem 5.4]{pelayo_two_tori}, the only four-manifolds that support a symplectic $T^2$ action, with a Hamiltonian subcircle action, are $S^2 \times \mathbb{T}^2$ and symplectic toric manifolds. Since symplectic toric manifolds are simply-connected, $S^2 \times \Sigma_g$ cannot be diffeomorphic to a symplectic toric manifold. Since $S^2 \times \Sigma_g$ is also not diffeomorphic to $S^2 \times \mathbb{T}^2$, it follows that the Hamiltonian action described above does not extend to a symplectic $T^2$ action.
\end{example}
\begin{example}\label{example_s2_s2_blown_up}
    Let $S^1$ act on $(S^2 \times S^2, dh_1 \wedge d\theta_1 + dh_2 \wedge d\theta_2)$ by rotating the left sphere at speed $1$. This is a Hamiltonian action, generated by the height function $H = h_1$ on the left sphere. Performing three equivariant blow ups of the same size, at points in the bottom fixed sphere, we obtain a Hamiltonian $S^1$ action on a simply-connected $4$-manifold. This action does not extend to a Hamiltonian $T^2$ action by \cite[Proposition 5.22]{karshon_periodic} since its corresponding Karshon graph has three fixed points with the same momentum value (see the related \cite[Figure 16]{karshon_periodic} and \cite[Example 7.4]{karshon_periodic}). Because the manifold is simply-connected, every symplectic torus action is Hamiltonian, and thus the action does not extend to any symplectic $T^2$ action. As before, the action is not almost isotropy-maximal.
\end{example}
As mentioned in the introduction, a torus action satisfies
\begin{equation}\label{inequality_isotropy}
    \dim T_x \le \dim M - \dim T
\end{equation}
at every point $x \in M$, and it follows that an isotropy-maximal action cannot be extended to a higher dimensional torus action. In the following remark, we briefly explain why Equation \eqref{inequality_isotropy} is always satisfied.
\begin{remark}\label{isotropy_maximal_inequality_remark}
    Let $T$ act on a manifold $M$, and let $x \in M$ be a point. By the slice theorem, a neighborhood of the orbit $T \cdot x$ is equivariantly diffeomorphic to the associated bundle $T \times_{T_x} \mathbb{R}^k$, where 
    \begin{equation}\label{some_dim_equation}
        k = \dim M - \dim T + \dim T_x.
    \end{equation}
    By the assumption that the action is effective, \cite[Corollary B.44]{ggk_book} implies that the linear representation of $T_x$ on $\mathbb{R}^k$ is effective, and therefore the dimension of $T_x$ is bounded from above by $k/2$. Then, Equation \eqref{inequality_isotropy} follows from this bound together with Equation \eqref{some_dim_equation}.
\end{remark}

\section{Almost isotropy-maximal actions have coisotropic orbits}\label{maximal_is_coisotropic_proof_section}
In this section, we prove Proposition \ref{maximal_is_coisotropic_proposition}, which states that almost isotropy-maximal symplectic torus actions have coisotropic principal orbits.
\begin{proof}[Proof of Proposition \ref{maximal_is_coisotropic_proposition}]
    By assumption, the action is almost isotropy-maximal, hence there exists some point $x \in M$ that satisfies
    \begin{equation}\label{almost_maximal_bound_equation}
        \dim M \le\dim T + \dim T_x + 1,
    \end{equation}
    where $T_x$ is the stabilizer of $x$.
    
    We denote the rank of $\sigma$ by $d$. By the definition of $\sigma$, the restriction of $\omega$ to the tangent space $T_x (T \cdot x)$ of the orbit $T \cdot x$ at the point $x$ has rank $d$. Therefore, $d$ is bounded by the dimension of the orbit $T \cdot x$, and we have
    \begin{equation}\label{bound_on_dim_t}
        d \le \dim (T \cdot x) = \dim T - \dim T_x.
    \end{equation}
    
    Assume by contradiction that some principal orbit $\mathcal{O}$, passing through some point $y \in M$, is not coisotropic. Since the orbit is principal, its dimension is $\dim T$. Again, by the definition of $\sigma$, the rank of $\omega|_{T_y\mathcal{O}}$ is $d$. By the assumption that $\mathcal{O}$ is not coisotropic, we have
    \begin{equation*}
        \dim T = \dim(\mathcal{O}) \le \frac{\dim M + d}{2} - 1,
    \end{equation*}
    which together with Equation \eqref{almost_maximal_bound_equation} gives
    \begin{equation*}
        \dim T \le \dim T_x + d - 1.
    \end{equation*}
    Combining it with Equation \eqref{bound_on_dim_t}, we deduce that
    \begin{equation*}
        \dim T_x + d \le \dim T_x + d - 1,
    \end{equation*}
    and this is a contradiction. It follows that the principal orbits of the action are coisotropic.
\end{proof}

\section{Coisotropic actions that are almost isotropy-maximal}\label{lemma_proof_section}
In this section, we prove Lemma \ref{which_coisotropic_are_maximal_lemma}, which gives simple conditions on when a symplectic torus action with coisotropic principal orbits is almost isotropy-maximal or isotropy-maximal.
\begin{proof}[Proof of Lemma \ref{which_coisotropic_are_maximal_lemma}]
    By \cite[Lemma 2.3]{duistermaat_pelayo}, we have the following equality
    \begin{equation*}
        \dim M = \dim T + \dim \ker(\sigma).
    \end{equation*}
    Substituting in Equation \eqref{maximal_equation}, we see that the action is almost isotropy-maximal if and only if there exists a point $x \in M$ such that
    \begin{equation}\label{first_equality}
        \dim T_x \ge \dim\ker(\sigma) - 1,
    \end{equation}
    and it is isotropy-maximal if and only if there exists a point $x \in M$ such that
    \begin{equation}\label{first_and_a_half_equality}
        \dim T_x = \dim \ker(\sigma).
    \end{equation}
    By \cite[Lemma 2.2]{duistermaat_pelayo}, for every point $x \in M$, the Lie algebra $\mathfrak{t}_x$ of the stabilizer $T_x$ is contained in $\ker(\sigma)$, so we have
    \begin{equation}\label{second_equality}
        \dim T_x \le \dim \ker(\sigma).
    \end{equation}
    Let $T_h$ be the unique maximal subtorus of $T$ which acts in a Hamiltonian fashion (see \cite[Lemma 3.6]{duistermaat_pelayo}). Since the action has coisotropic principal orbits, we can choose a complementary subtorus $T_f$ which acts freely (see \cite[Section 5.2]{duistermaat_pelayo}). 
    Because $T_f$ acts freely, and since it is complementary to $T_h$ in $T$, we have
    \begin{equation}\label{third_equality}
        \dim T_x \le \dim T - \dim T_f = \dim T_h,
    \end{equation}
    for every point $x \in M$. On the other hand, by the Atiyah-Guillemin-Sternberg's convexity theorem (\cite{atiyah_convexity}, \cite{gs_convexity}), there exists some point $y \in M$ which is fixed by the action of $T_h$, so we have
    \begin{equation}\label{fourth_equality}
        \dim T_y \ge \dim T_h.
    \end{equation}
    It follows by Equations \eqref{third_equality} and \eqref{fourth_equality} that $\dim T_y = \dim T_h$, and that this is the maximal dimension of a stabilizer of the action. Thus, using also Equation \eqref{first_equality}, the action is almost isotropy-maximal if and only if we have
    \begin{equation*}
        \dim \mathfrak{t}_h = \dim T_h \ge \dim \ker(\sigma) - 1,
    \end{equation*}
    where $\mathfrak{t}_h$ is the Lie algebra of $T_h$. Furthermore, by Equation \eqref{first_and_a_half_equality} and Equation \eqref{second_equality}, it is isotropy-maximal if and only if we have
    \begin{equation*}
        \dim \mathfrak{t}_h = \dim T_h = \dim \ker(\sigma),
    \end{equation*}
    if and only if $\mathfrak{t}_h = \ker(\sigma)$.
\end{proof}

\section{Classification of isotropy-maximal actions}\label{proofs_of_class_section}
In this section, we prove Theorems \ref{maximal_uniqueness_theorem} and \ref{maximal_existence_theorem}, giving a simple description of the classification of isotropy-maximal symplectic torus actions.

\begin{proof}[Proof of Theorem \ref{maximal_uniqueness_theorem}]
    Let $(M, \omega)$ a symplectic manifold, equipped with an isotropy-maximal symplectic torus action. By Lemma \ref{which_coisotropic_are_maximal_lemma}, we have that $\mathfrak{t}_h = \ker(\sigma)$. Therefore, the second invariant $T_h$ can be read from the first invariant $\sigma$. Moreover, it follows that the group $N :=(\ker(\sigma)/\mathfrak{t}_h)^*$ is trivial, and therefore the fourth invariant $P$ and the fifth invariant $c$ are also trivial by definition (cf. Remark \ref{fifth_invariant_remark}). Since $P$ is trivial, then the sixth invariant $\bar \tau$ is also trivial by definition.
    
    Now, let $(M', \omega')$ be another symplectic manifold, equipped with an isotropy-maximal symplectic torus action. By Theorem \ref{almost_uniqueness_theorem}, the two spaces are equivariantly symplectomorphic if and only if the six invariants agree. By the above discussion, it follows that $(M, \omega)$ and $(M', \omega')$ are equivariantly symplectomorphic if and only if their invariants $\sigma$ and $\Delta$ agree.
\end{proof}
\begin{proof}[Proof of Theorem \ref{maximal_existence_theorem}]
    Let $T \cong T_h \times T_f, \sigma, \Delta \subset \mathfrak{t}_h^*,$ and $(M_h, \omega_h, T_h)$ be as in the statement of the theorem. 
    
    As mentioned in the statement of the theorem, $\sigma$ induces a unique invariant symplectic form $\omega_f$ on $T_f$. More explicitly, $\omega_f$ is the unique invariant two-form on $T_f$ that satisfies
    \begin{equation*}
        \exp^*\omega_f = j^* \sigma, 
    \end{equation*}
    where $\exp:\mathfrak{t}_f \rightarrow T_f$ is the exponential map, and $j:\mathfrak{t}_f \rightarrow \mathfrak{t}$ is the embedding given by the isomorphism $T \cong T_h \times T_f$. Because $T_f$ is complementary to $T_h$, and since $\ker(\sigma) = \mathfrak{t}_h$, we have $\dim(T_f) = \rank(\sigma) = \rank((\omega_f)_y)$ at any point $y \in T_f$. Therefore, $\omega_f$ is non-degenerate and thus symplectic.
    
    It follows that the free action of $T_f$ on itself is an isotropy-maximal symplectic torus action (see also Example \ref{maximal_example_torus}). As mentioned in Example \ref{maximal_example_delzant}, the action of $T_h$ on $M_h$ is also an isotropy-maximal symplectic torus action. Thus, as in Example \ref{product_example}, the action of $T \cong T_h \times T_f$ on the product space $(M_h \times T_f, \omega_h + \omega_f)$, given by
    \begin{equation*}
        (t, t') \cdot (x,y) = (t\cdot x, t'\cdot y)
    \end{equation*}
    is an isotropy-maximal symplectic torus action.

    Since $T_h$ is the maximal subtorus that acts in a Hamiltonian fashion, the Delzant polytope of the product space is the momentum image of the $T_h$ action, which is $\Delta$ by assumption. Finally, by the construction of the symplectic form $\omega_f$, the unique antisymmetric bilinear form on $\mathfrak{t}$ defined by Equation \eqref{bilinear_form_equation} is equal to $\sigma$.
\end{proof}

\section{Proof of Theorem \ref{almost_existence_theorem}}\label{proof_of_almost_existence_theorem_section}
In this section, we prove Theorem \ref{almost_existence_theorem}. One approach for proving Theorem \ref{almost_existence_theorem}, is to follow Duistermaat and Pelayo's construction of actions with coisotropic orbits (see \cite[Section 9.2]{duistermaat_pelayo}), and show that if $\dim \mathfrak{t}_h = \dim \ker(\sigma) - 1$, then the space that they construct can indeed be described as a product of a symplectic toric manifold and a torus. However, their construction is very general, so we decided to prove the theorem directly from \cite[Lemma 7.5]{duistermaat_pelayo} and Lemma \ref{linearizing_torus_action_lemma} below, for clarity.
\begin{proof}[Proof of Theorem \ref{almost_existence_theorem}]
    Let $(\sigma, T_h, \Delta, P, c, \bar\tau)$ be a list of ingredients which satisfies $\dim \mathfrak{t}_h = \dim\ker(\sigma) - 1$, and let $(M_h, \omega_h, T_h)$ be a symplectic toric manifold with momentum image $\Delta$.
    By \cite[Theorem 9.6]{duistermaat_pelayo}, there exists a symplectic manifold $(M, \omega)$ with a symplectic $T$ action, whose principal orbits are coisotropic, and whose Duistermaat-Pelayo invariants are equal to $\sigma, T_h, \Delta, P, c,$ and $\bar\tau$.
    
    By Lemma \ref{which_coisotropic_are_maximal_lemma}, because $\dim \mathfrak{t}_h = \dim \ker(\sigma) - 1$, the $T$ action on $M$ is strictly almost isotropy-maximal. We wish to show that it can be described as a product space as in the statement.
    
    By definition, the subspace $N := (\ker(\sigma)/\mathfrak{t}_h)^*$ is one-dimensional. As mentioned in Section \ref{results_section}, the fifth invariant can be identified with an antisymmetric map $c:N \times N \rightarrow \ker(\sigma)$, and therefore it must be the zero map since $\dim N = 1$ (cf. Remark \ref{fifth_invariant_remark}). Another argument is that the Chern class of the principal $T$-bundle $\mu_{\free} :M_{\free} \rightarrow \interior(\Delta) \times N/P$ is an element of $H^2(\interior(\Delta) \times N/P, \mathfrak{t})$, and because $\interior(\Delta) \times N/P$ is homotopy-equivalent to $S^1$, its second cohomology vanishes and therefore the Chern class vanishes.
    
    Let $T_h$ be the unique maximal subtorus of $T$ acting in a Hamiltonian fashion, and let $T_f$ be a complementary subtorus to $T_h$ in $T$. Since $c$ is trivial, it follows by \cite[Lemma 7.5]{duistermaat_pelayo} that $(M, \omega)$ is equivariantly symplectomorphic to a product of a symplectic toric manifold $(M_h, \omega_h, T_h)$ and a symplectic manifold $(M_f, \omega_f, T_f)$ equipped with a free $T_f$ action.
    
    Since the action is strictly almost isotropy-maximal, we have a point $x \in M$ which satisfies
    \begin{equation}
        \dim T + \dim T_x = \dim M - 1.
    \end{equation}
    Because $T_f$ acts freely, and $T_h$ is a Hamiltonian action, the left-hand side of the equation is equal to $2\dim T_h + \dim T_f$ at fixed points of $T_h$, and this is the maximum value possible over all points in $M$. Moreover, because $2 \dim T_h = \dim M_h$ and $\dim M_h + \dim M_f = \dim M$, we deduce that $\dim T_f = \dim M_f - 1$.

    We are left with showing that there exists a torus $\mathbb{T}$ with a $\mathbb{T}$-invariant symplectic form $\omega_{\mathbb{T}}$, and a codimension one subtorus $K \subset \mathbb{T}$, such that the action of $T_f$ on $(M_f, \omega_f)$ is equivariantly symplectomorphic to the natural action of $K$ on $(\mathbb{T}, \omega_{\mathbb{T}})$. This follows from Lemma \ref{linearizing_torus_action_lemma} below.
\end{proof}
\begin{lemma}\label{linearizing_torus_action_lemma}
    Let a torus $T^{2n-1}$ act freely and symplectically on a symplectic manifold $(M^{2n}, \omega)$. Then there exists a torus $\mathbb{T}^{2n}$, a subtorus $K^{2n-1} \subset \mathbb{T}^{2n}$, and a $\mathbb{T}^{2n}$-invariant symplectic form $\omega_{\mathbb{T}}$ on $\mathbb{T}^{2n}$, such that the action of $T$ on $(M, \omega)$ is equivariantly symplectomorphic to the natural action of $K^{2n-1}$ on $(\mathbb{T}^{2n}, \omega_{\mathbb{T}})$.
\end{lemma}
\begin{proof}
    First, since $\dim T = \dim M - 1$, then $M \rightarrow M/T$ defines a principal torus bundle over a circle. Let $\mathfrak{t}$ be the Lie algebra of $T$. Then the Chern class of $M \rightarrow M/T$ is a cohomology class in $H^2(S^1, \mathfrak{t}) = 0$, and therefore the bundle is trivial. It follows that $M$ is diffeomorphic to a torus $\mathbb{T}^{2n}$, and we can choose coordinates on $M \cong \mathbb{T}^{2n}$ and $T^{2n-1}$, such that the action is given by
    \begin{equation*}
        (t_1,...,t_{2n-1}) \cdot (x_1,...,x_{2n}) = (x_1 + t_1,...,x_{2n-1} + t_{2n-1}, x_{2n}),
    \end{equation*}
    where the coordinates $x_1,...,x_{2n}$ on the torus $\mathbb{T}^{2n}$ and the coordinates $t_1,...,t_{2n-1}$ on the acting torus $T^{2n-1}$ are all given as values in $\mathbb{R}/\mathbb{Z}$. In particular, from these coordinates we can identify $T^{2n-1}$ with a subgroup $K^{2n-1}$ of $\mathbb{T}^{2n}$.
    
    The symplectic form $\omega$ on $\mathbb{T}^{2n} \cong M$ can then be written as
    \begin{equation*}
        \omega(x) = \sum_{1 \le i < j \le 2n} f_{i,j}(x) dx_i \wedge dx_j.
    \end{equation*}
    By the $T^{2n-1}$-invariance of $\omega$, the functions $f_{i,j}$ only depend on the $x_{2n}$ coordinate, and we will write $f_{i,j}(x_{2n})$ from now on. For every point $y \in \mathbb{T}^{2n}$, and integers $1 \le i < j \le 2n$, let $\Sigma_{i,j}(y)$ be the 2-dimensional submanifold of $\mathbb{T}^{2n}$ given by
    \begin{equation*}
        \Sigma_{i,j}(y) := \{ (x_1,...,x_{2n}) \in \mathbb{T}^{2n}|\ \forall k \notin \{i,j\} : x_k = y_k \}.
    \end{equation*}
    Since $\omega$ is closed, its integral over the submanifold $\Sigma_{i,j}(y)$ is independent of the choice of $y$, by Stokes theorem. We define the constants $a_{i,j}$ by
    \begin{equation*}
        a_{i,j} = \int_{\Sigma_{i,j}(y)} \omega = \int_0^1\int_0^1 f_{i,j}(x_{2n}) dx_i dx_j,
    \end{equation*}
    and note that if $1 \le i < j \le 2n-1$, then $a_{i, j} = f_{i, j}(x_{2n})$ for every $x_{2n}$, so $f_{i, j}$ is a constant function.
    The cohomology class $[\omega]$ of the symplectic form can be written as
    \begin{equation*}
        [\omega] = \sum_{1 \le i < j \le 2n} a_{i,j} [dx_i \wedge dx_j].
    \end{equation*}
    We define a $\mathbb{T}^{2n}$-invariant symplectic form on $\mathbb{T}^{2n}$ by
    \begin{equation*}
        \omega_{\mathbb{T}}(x) = \sum_{1 \le i < j \le 2n} a_{i,j} dx_i \wedge dx_j,
    \end{equation*}
    and obviously we have $[\omega] = [\omega_{\mathbb{T}}]$. The difference between the symplectic forms $\omega_{\mathbb{T}}$ and $\omega$ is given by
    \begin{equation*}
        \omega_{\mathbb{T}} - \omega = \sum_{1 \le i \le 2n-1} (a_{i, 2n} - f_{i, 2n}(x_{2n})) dx_i \wedge dx_{2n}.
    \end{equation*}
    In particular, because $a_{i, 2n} = \int_0^1 f_{i, 2n} (x_{2n}) dx_{2n}$, then the functions
    \begin{equation*}
        b_i(s) := \int_0^{s} (f_{i, 2n} (x_{2n}) - a_{i, 2n}) dx_{2n}
    \end{equation*}
    are well defined for $s \in \mathbb{R}/\mathbb{Z}$. Moreover, they satisfy $\frac{\partial b_i}{\partial x_{2n}}(x_{2n}) = f_{i, 2n} (x_{2n}) - a_{i, 2n}$. Therefore, the one-form $\beta$ defined by
    \begin{equation*}
        \beta := \sum_{1 \le i \le 2n - 1} b_i(x_{2n}) dx_i
    \end{equation*}
    satisfies $d \beta = \omega_{\mathbb{T}} - \omega$. For every $t \in [0,1]$ we define $\omega_t = (1-t) \omega + t\omega_{\mathbb{T}}$. We will show below that $\omega_t$ is non-degenerate for every $t$. It then follows, by applying Moser's trick, that the flow $\Psi_t$ given by integrating the unique vector field $Z_t$ which solves the equation
    \begin{equation*}
        \iota_{Z_t} \omega_t = -d\beta,
    \end{equation*}
    satisfies $\Psi_1^*\omega_{\mathbb{T}} = \omega$, i.e., $\Psi_1$ is a symplectomorphism between $(\mathbb{T}^{2n}, \omega)$ and $(\mathbb{T}^{2n}, \omega_{\mathbb{T}})$. Furthermore, $\beta$ and $\omega_t$ are $T^{2n-1}$-invariant, and therefore the symplectomorphism $\Psi_1$ is $T^{2n-1}$-equivariant, and this finishes the proof.

    Therefore, we are left with showing that $\omega_t$ is indeed non-degenerate for every $t \in [0,1]$. For every $x \in \mathbb{T}^{2n}$, let $V_x$ be the $2n-1$ dimensional subspace of $T_x \mathbb{T}^{2n}$ generated by the tangent vectors $\frac{\partial}{\partial x_1},...,\frac{\partial}{\partial x_{2n-1}}$. We identify all of these subspaces with $\mathbb{R}^{2n-1}$ and call them all $V$. Restricting $\omega_x$ and $(\omega_{\mathbb{T}})_x$ to $V_x$, we get the same antisymmetric bilinear form $\omega_x|_V = (\omega_{\mathbb{T}})_x|_V$ on $V$, and it has rank $2n-2$. The restriction $\omega_x|_V$ is independent of $x \in M$, hence there exists a vector field
    \begin{equation*}
        u := \sum_{1 \le i \le 2n-1}u_i\frac{\partial}{\partial x_i},
    \end{equation*}
    which satisfies $\omega_y(u, v) = (\omega_{\mathbb{T}})_y(u, v) = 0$, for every $v \in V$ and $y \in M$.
    
    Since $\omega$ and $\omega_{\mathbb{T}}$ are non-degenerate, we have $\omega_y(u, \frac{\partial}{\partial x_{2n}}) \ne 0$ and $(\omega_{\mathbb{T}})_y(u, \frac{\partial}{\partial x_{2n}}) \ne 0$ at every point $y \in M$. It follows that $\omega_y(u, \frac{\partial}{\partial x_{2n}})$ does not change sign when we change the $x_{2n}$-coordinate of $y$. In coordinates, we have
    \begin{align*}
            \omega_y(u, \frac{\partial}{\partial x_{2n}}) = & \sum_{1 \le i \le 2n-1} u_i f_{i, 2n}(x_{2n}) \\
            (\omega_{\mathbb{T}})_y(u, \frac{\partial}{\partial x_{2n}}) = & \sum_{1 \le i \le 2n-1} u_i a_{i, 2n}
    \end{align*}
    and because $a_{i, 2n} = \int_0^1 f_{i, 2n} (x_{2n}) dx_{2n}$, then both $\omega_y(u, \frac{\partial}{\partial x_{2n}})$ and $(\omega_{\mathbb{T}})_y(u, \frac{\partial}{\partial x_{2n}})$ have the same sign. It follows that for every $t \in [0,1]$ we have
    \begin{equation*}
        ((1-t)\omega + t\omega_{\mathbb{T}})_y(u, \frac{\partial}{\partial x_{2n}}) \ne 0,
    \end{equation*}
    and therefore $(1-t)\omega + t\omega_{\mathbb{T}}$ is non-degenerate for every $t \in [0,1]$ and $y \in M$.
\end{proof}

\printbibliography

\end{document}